\UseRawInputEncoding
\documentclass[12pt,reqno]{amsart}
\usepackage{amssymb,amsmath,graphicx,
	amsfonts,euscript}
\usepackage{color}
\usepackage{mathrsfs}
\UseRawInputEncoding
\theoremstyle{plain}

\newtheorem{theorem}{Theorem}[section]
\newtheorem{lemma}[theorem]{Lemma}

\newtheorem{remark}[theorem]{Remark}
\numberwithin{equation}{section}

\begin{document}

\title[time-periodic solution for the fractal Burgers equation]{Uniqueness and asymptotic stability of time-periodic solution for the fractal Burgers equation}

\author[Yong Zhang, Fei Xu, Fengquan Li]{Yong Zhang, Fei Xu, Fengquan Li}

\address[Yong Zhang]{ School of Mathematical Sciences, Dalian University
of Technology, Dalian, 116024, China} \email{18842629891@163.com}

\address[Fei Xu]{School of Mathematical Sciences, Dalian University
of Technology, Dalian, 116024, China}

\address[Fengquan Li]{School of Mathematical Sciences, Dalian University
of Technology, Dalian, 116024, China}

\begin{abstract}
The paper is concerned with the time-periodic (T-periodic) problem of the fractal Burgers equation with a T-periodic force on the real line. Based on the Galerkin approximates and Fourier series (transform) methods, we first prove the existence of T-periodic solution to a linearized version. Then, the existence and uniqueness of T-periodic solution to the nonlinear equation are established by the contraction mapping argument. Furthermore, we show that the unique T-periodic solution is asymptotically stable. This analysis, which is carried out in energy space $ H^{1}(0,T;H^{\frac{\alpha}{2}}(R))\cap L^{2}(0,T;\dot{H}^{\alpha})$ with $1<\alpha<\frac{3}{2}$, extends the T-periodic viscid Burgers equation in \cite{Chen} to the T-periodic fractional case.
	\end{abstract}
\maketitle

\section{Introduction and main result}

In this paper, we are interested in T-periodic motions for the following forced Burgers equation of fractal order
\begin{equation}\label{e11}
u_{t}+uu_{x}+\kappa\Lambda^{\alpha}u=f,\quad (t,x)\in [0,T]\times R,
\end{equation}
where $\kappa>0$ is the viscosity coefficient and $1<\alpha<\frac{3}{2}$ is the dissipation index. The fractional dissipation term $\Lambda^{\alpha}u=(-\partial_{xx})^{\frac{\alpha}{2}}u$ is defined by the Fourier transform
$$
\widehat{\Lambda^{\alpha}u}(\xi)=|\xi|^{\alpha}\hat{u}(\xi)
$$
and $f(t,x)$ is a T-periodic external force, i.e. $f(t+T,x)=f(t,x)$.

The viscid Burgers (or Burgers) equation is one of the most simplest but important partial differential equation to model Navier-Stokes (or Euler) equation's nonlinearity. In recent years, there has been a great deal of interest in using the fractional dissipation to describe diverse physical phenomena, such as anomalous diffusion and quasi-geostrophic flows, turbulence and water waves and molecular dynamics (see \cite{Bouchard,Caffarelli1,Caffarelli2,Constantin} and the references therein). The main purpose of this paper is to seek the T-periodic solution, i.e. $u(t,x)=u(t+T,x)$, to \eqref{e11}, which is expected to be unique and stable. Besides, the results and techniques presented here may be applied to a wider class of equations including quasi-geostrophic equation, Boussinesq system and so on.

Before stating our main results, let's mention about some important work on Burgers equation involving fractional dissipation or periodic behavior. To our knowledge, the well-posedness of fractal Burgers equation \eqref{e11} with $f=0$ depends heavily on the index $\alpha$ (see \cite{Dong,Kiselev,Miao}). In the supercritical dissipative case ($0<\alpha<1$), the equation is locally well-posed and its solution develops gradient blow-up in finite time. In the critical dissipative case ($\alpha=1$) and subcritical dissipative case ($1<\alpha<2$), such singularity does not appear so that the solution always exists globally in time. Besides, the results on the global regularizing effects in the subcritical case and the non-uniqueness of weak solutions in the supercritical case were also established in \cite{Droniou} and \cite{Alibaud}, respectively. Recently, the analyticity and large time behavior in the critical case were given in \cite{Iwabuchi}.

On the other hand, the existence of the time-periodic solution of \eqref{e11} with $\alpha=2$ was obtained for $(t,x)\in [0,T]\times [-1,1]$ in \cite{Chen}, where the authors also proved that this time-periodic solution was unique and asymptotically stable in the $H^{1}$ sense under an additional smallness condition on $f(t,x)$. The non-autonomously forced Burgers equation with periodic and Dirichlet boundary conditions was also studied in \cite{Piotr}. In addition, many authors concentrated on the existence and uniqueness of time-periodic solutions and investigated how the solutions near the time-periodic solutions behave as time goes on, see \cite{Galdi1,Galdi2,Hsia,Kobayashi,Kyed,Cheng}.

Although the Burgers equation is fundamental, there is little information on the T-periodic solution to (\eqref{e11}) with $\alpha\neq2$. In some physical models, the T-periodic solution plays the key role in describing natural phenomena. This is a main motivation of the study. Throughout this paper, we denote $f\lesssim g$ when $f\leq cg$ for some constant $c>0$ and let $H^{s}(R)$ and $\dot{H}^{s}(R)$ be the nonhomogeneous and homogeneous Sobolev spaces. Our first result in this paper is about the existence and uniqueness of T-periodic solution with finite energy to \eqref{e11}, which can be stated as follows.
\begin{theorem}\label{thm1.1}
Assume that $f\in X=H^{1}(0,T;\dot{H}^{-\frac{\alpha}{2}}(R))\cap L^{2}(0,T;L^{2}(R))\cap L^{2}(0,T;\dot{H}^{-\alpha}(R))$ with $1<\alpha<\frac{3}{2}$ and the force $f(t,x)$ satisfies
\begin{equation*}
f(0)=f(T) ~~in~~ \dot{H}^{-\frac{\alpha}{2}}(R)~~and~~\|f\|_{X}<C(\kappa,\alpha),
\end{equation*}
where $C(\kappa,\alpha)=\min\{ \frac{2\kappa^{2}\sqrt{2(3-2\alpha)}}{9(\kappa+1)^{2}(\sqrt{2(3-2\alpha)}+1)}, \frac{\kappa^{2}}{6(\kappa+1)^{2}} \}$, then there exists a unique T-periodic strong solution
$$
u\in H^{1}(0,T;H^{\frac{\alpha}{2}}(R))\cap L^{2}(0,T;\dot{H}^{\alpha})
$$
to problem (\eqref{e11}) satisfying $u(0)=u(T)$ in $H^{\frac{\alpha}{2}}(R)$ and
\begin{equation}\label{e13}
\|u\|_{H^{1}(0,T;H^{\frac{\alpha}{2}}(R))}+\|\Lambda^{\alpha}u\|_{L^{2}(0,T;L^{2}(R))}\leq 3(1+\frac{1}{\kappa})\|f\|_{X}.
\end{equation}
\end{theorem}
\begin{remark}
It's known that the condition $\alpha\geq1$ is necessary to ensure the global existence of strong solution to \eqref{e11}. Thus, it's reasonable to expect that this result holds for $1\leq\alpha<2$. However, the techniques of Sobolev embedding and Fourier analysis prevent us improving the range of $\alpha$ in Theorem \ref{thm1.1}. This is an interesting question of future research.
\end{remark}
Besides, we also prove the unique T-periodic strong solution of \eqref{e11} obtained in Theorem \ref{thm1.1} is asymptotically stable in the following.
\begin{theorem}\label{thm1.3}
Suppose $f(t,x)$ satisfies the assumption in Theorem \ref{thm1.1}, then the T-periodic solution $u_{T}(t,x)$ of \eqref{e11} is asymptotic stability in the sense: assume $u(t,x)$ be a viscosity weak solution of \eqref{e11} with initial data $u_{0}\in L^{2}(R)$, then
\begin{equation}\label{e14}
\|u-u_{T}\|_{L^{\infty}(0,T;L^{2}(R))\cap L^{2}(0,T;\dot{H}^{\frac{\alpha}{2}}(R))}\leq \|u_{0}-u_{T}(0)\|_{L^{2}(R)}.
\end{equation}
\end{theorem}

The paper is organized as follows. In section 2, we are devoted to deal with a linearized version of problem \eqref{e11} by using the Fourier expansion with respect to time variable and Galerkin approximates with respect to spatial variable. In section 3, we apply a contraction mapping argument to obtain the existence and uniqueness of T-periodic solution for the nonlinear problem \eqref{e11}. In the last section, we will show that the T-periodic solution is asymptotically stable.

\section{The linearized problem}

In this section, we consider the linearized problem associated with \eqref{e11}
\begin{equation}\label{e21}
u_{t}+\kappa\Lambda^{\alpha}u=f,\quad for~~(t,x)\in[0,T]\times R,
\end{equation}
where $\kappa>0$ and $1<\alpha<\frac{3}{2}$. Inspired by the work \cite{Galdi1,Chen}, we mainly use the Fourier and Galerkin methods to obtain the following result:
\begin{theorem}\label{thm2.1}
Assume that $f\in X=H^{1}(0,T;\dot{H}^{-\frac{\alpha}{2}}(R))\cap L^{2}(0,T;L^{2}(R))\cap L^{2}(0,T;\dot{H}^{-\alpha}(R))$ with $f(0)=f(T)$ in $\dot{H}^{-\frac{\alpha}{2}}(R)$, then there exists a unique T-periodic strong solution
$$
u\in H^{1}(0,T;H^{\frac{\alpha}{2}}(R))\cap L^{2}(0,T;\dot{H}^{\alpha})
$$
to problem (\ref{e21}) satisfying $u(0)=u(T)$ in $H^{\frac{\alpha}{2}}(R)$ and
\begin{equation}\label{e22}
\|u\|_{H^{1}(0,T;H^{\frac{\alpha}{2}}(R))}+\|\Lambda^{\alpha}u\|_{L^{2}(0,T;L^{2}(R))}\leq (1+\frac{1}{\kappa})\|f\|_{X}.
\end{equation}
\end{theorem}

\begin{proof}
For a clear presentation, we divide the proof into the following three steps. We first try to find a function $u\in H^{1}(0,T;\dot{H}^{\frac{\alpha}{2}}(R))$ with $u_{t}\in L^{2}(0,T;L^{2}(R))$, which satisfies
\begin{equation}\label{e23}
\int^{T}_{0}\int_{R}u_{t}\varphi dxdt+\kappa\int^{T}_{0}\int_{R}\Lambda^{\frac{\alpha}{2}}u\Lambda^{\frac{\alpha}{2}}\varphi dxdt=\int^{T}_{0}\int_{R}f\varphi dxdt
\end{equation}
for all $\varphi\in C^{\infty}_{0,T}((0,T)\times R)=:\{\varphi\in C^{\infty}_{0}|~ \varphi(0)=\varphi(T)\}$. Secondly, we will show that there hold
\begin{equation}\label{e24}
u(0)=u(T) \quad in ~~H^{\frac{\alpha}{2}}(R),
\end{equation}
\begin{equation}\label{e25}
\|\Lambda^{\frac{\alpha}{2}}u\|_{L^{2}(0,T;L^{2}(R))}\leq \frac{1}{\kappa}\|f\|_{L^{2}(0,T;\dot{H}^{-\frac{\alpha}{2}}(R))},
\end{equation}
\begin{equation}\label{e26}
\|u_{t}\|_{L^{2}(0,T;L^{2}(R))}\leq \|f\|_{L_{2}(0,T;L^{2}(R))},
\end{equation}
\begin{equation}\label{e27}
\|\Lambda^{\frac{\alpha}{2}}u_{t}\|_{L^{2}(0,T;L^{2}(R))}\leq \frac{1}{\kappa} \|f_{t}\|_{L^{2}(0,T;\dot{H}^{-\frac{\alpha}{2}}(R))}.
\end{equation}
Finally, we use the Fourier transform method to make up the regularity of the solution, that is to say, there also hold
\begin{equation}\label{e28}
\|u\|_{L^{2}(0,T;L^{2}(R))}\leq \frac{1}{\kappa}\|f\|_{L^{2}(0,T;\dot{H}^{-\alpha}(R)))}
\end{equation}
and
\begin{equation}\label{e29}
\|\Lambda^{\alpha}u\|_{L^{2}(0,T;L^{2}(R))}\leq \frac{1}{\kappa} \|f\|_{L^{2}(0,T;L^{2}(R))}.
\end{equation}
{\bf Step 1: Existence and uniform estimates for periodic approximating sequence $\{u_{nm}\}$}

Let $\{\psi_{j}\}\subset C^{\infty}_{0}(R)$ be a basis of $H^{\frac{\alpha}{2}}(R)$ such that
\begin{equation}\label{e210}
\int_{R}\Lambda^{\frac{\alpha}{2}}\psi_{j}\Lambda^{\frac{\alpha}{2}}\psi_{k}dx=\delta_{jk},\quad \forall j,k\in N.
\end{equation}
Since we aim to find a time-periodic solution, it's natural to assume that the solution can be written as
$$
u=\lim_{n\rightarrow +\infty}u_{n}(t,x),
$$
where
\begin{equation}\label{e211}
u_{n}(t,x)=\sum_{l=1}^{n}a(n)_{l}(x)sin(\frac{2l\pi t}{T})+\sum_{l=0}^{n}b(n)_{l}(x)cos(\frac{2l\pi t}{T}).
\end{equation}
For each fixed $n\in N$, we first construct a Galerkin approximating sequence $\{u_{nm}\}_{m\in N}$ satisfying \eqref{e21} in the sense of \eqref{e23} by
\begin{equation}\label{e212}
u_{nm}(t,x)=\sum_{l=1}^{n}[\sum_{k=0}^{m}a(n,m)_{lk}\psi_{k}(x)]sin(\frac{2l\pi t}{T})+\sum_{l=0}^{n}[\sum_{k=0}^{m}b(n,m)_{lk}\psi_{k}(x)]cos(\frac{2l\pi t}{T}).
\end{equation}
It's necessary for us to prove the existence and uniqueness of approximating solutions $u_{nm}(t,x)$ (or the coefficients $a(n,m)_{lk}$ and $b(n,m)_{lk}$). To achieve the goal, we choose the test function $\varphi\in C^{\infty}_{0,T}((0,T)\times R)$ with the form
\begin{equation}\label{e213}
\varphi(t,x)=\sum_{i=1}^{n}\sum_{j=0}^{m}\psi_{j}(x)sin(\frac{2i\pi t}{T})+\sum_{i=0}^{n}\sum_{j=0}^{m}\psi_{j}(x)cos(\frac{2i\pi t}{T}).
\end{equation}
Taking \eqref{e212} and \eqref{e213} into \eqref{e23}, we obtain
\begin{align}\label{e214}
~&\int_{0}^{T}\int_{R}\partial_{t}u_{nm}\psi_{j}(x)sin(\frac{2i\pi t}{T})dxdt+\kappa\int_{0}^{T}\int_{R}\Lambda^{\frac{\alpha}{2}}u_{nm}\Lambda^{\frac{\alpha}{2}}\psi_{j}(x)sin(\frac{2i\pi t}{T})dxdt \nonumber\\
&=\int^{T}_{0}\int_{R}f\psi_{j}(x)sin(\frac{2i\pi t}{T})dxdt, \quad (i=1:n,~ j=0:m)
 \end{align}
and
\begin{align}\label{e215}
~&\int_{0}^{T}\int_{R}\partial_{t}u_{nm}\psi_{j}(x)cos(\frac{2i\pi t}{T})dxdt+\kappa\int_{0}^{T}\int_{R}\Lambda^{\frac{\alpha}{2}}u_{nm}\Lambda^{\frac{\alpha}{2}}\psi_{j}(x)cos(\frac{2i\pi t}{T})dxdt \nonumber\\
&=\int^{T}_{0}\int_{R}f\psi_{j}(x)cos(\frac{2i\pi t}{T})dxdt, \quad (i=0:n,~ j=0:m).
 \end{align}
Considering the follwing orthogonality properties of trigonometric functions
\begin{equation}
\begin{aligned}
~&\int^{T}_{0}sin(\frac{2l\pi t}{T})cos(\frac{2i\pi t}{T})dt=0, \quad \forall ~l,i\in N_{0}, \nonumber\\
& \int^{T}_{0}sin(\frac{2l\pi t}{T})sin(\frac{2i\pi t}{T})dt= \int^{T}_{0}cos(\frac{2l\pi t}{T})cos(\frac{2i\pi t}{T})dt=\frac{T}{2}\delta_{il},
 \end{aligned}
\end{equation}
we can reduce the \eqref{e214} and \eqref{e215} to
\begin{align}\label{e216}
~&-i\pi\sum_{k=0}^{m}b(n,m)_{ik}\int_{R}\psi_{k}(x)\psi_{j}(x)dx+\frac{\kappa T}{2}a(n,m)_{ij} \nonumber\\
&=\int^{T}_{0}\int_{R}f\psi_{j}(x)sin(\frac{2i\pi t}{T})dxdt, \quad (i=1:n,~ j=0:m)
 \end{align}
and
\begin{align}\label{e217}
~&i\pi\sum_{k=0}^{m}a(n,m)_{ik}\int_{R}\psi_{k}(x)\psi_{j}(x)dx+\frac{\kappa T}{2}b(n,m)_{ij} \nonumber\\
&=\int^{T}_{0}\int_{R}f\psi_{j}(x)cos(\frac{2i\pi t}{T})dxdt, \quad (i=0:n,~ j=0:m).
 \end{align}
For a fixed $i\in \{1,2,...,n\}$, we can write \eqref{e216} and \eqref{e217} into a algebraic equation set
\begin{equation}\label{e218}
Ay=F,
\end{equation}
where $A$ is a $(2m+2)\times(2m+2)$ matrix
\begin{equation}  \nonumber\\
\left[
\begin{array}{cc}
\frac{\kappa T}{2}\delta_{jk}  &  -i\pi\int_{R}\psi_{j}(x)\psi_{k-m-1}(x)\\
(k=0:m, ~j=0:m) & (k=m+1:2m+1, ~j=0:m)\\
~ & ~\\
i\pi\int_{R}\psi_{j-m-1}(x)\psi_{k}(x)  & \frac{\kappa T}{2}\delta_{(j-m-1)(k-m-1)} \\
(k=0:m, ~j=m+1:2m+1) &  (k=m+1:2m+1, ~j=m+1:2m+1)\\
\end{array}
\right],
\end{equation}
$y$ is a $(2m+2)\times 1$ vector
\begin{equation}  \nonumber\\
\left[
\begin{array}{c}
a(n,m)_{ij} \\
(j=0:m) \\
~ \\
b(n,m)_{i(j-m-1)}  \\
(j=m+1:2m+1) \\
\end{array}
\right]
\end{equation}
and $F$ is also a $(2m+2)\times 1$ vector
\begin{equation}  \nonumber\\
\left[
\begin{array}{c}
\int^{T}_{0}\int_{R}f\psi_{j}(x)sin(\frac{2i\pi t}{T})dxdt \\
(j=0:m) \\
~ \\
\int^{T}_{0}\int_{R}f\psi_{j-m-1}(x)cos(\frac{2i\pi t}{T})dxdt \\
(j=m+1:2m+1) \\
\end{array}
\right].
\end{equation}
It's obvious that $A$ is invertible, then it follows from \eqref{e218} that
$$
y=A^{-1}F
$$
can be uniquely determined.

Now we establish uniform estimates for the Galerkin approximating solutions $u_{nm}(t,x)$. Multiplying \eqref{e214} by $a(n,m)_{ij}$, multiplying \eqref{e215} by $b(n,m)_{ij}$ and adding them together yield
$$
\int^{T}_{0}\int_{R}\partial_{t}u_{nm}u_{nm}dxdt+\kappa\int^{T}_{0}\int_{R}|\Lambda^{\frac{\alpha}{2}}u_{nm}|^{2}dxdt=\int^{T}_{0}\int_{R}fu_{nm}dxdt,
$$
which also gives
\begin{equation}\label{e219}
\|\Lambda^{\frac{\alpha}{2}}u_{nm}(t)\|_{L^{2}(0,T;L^{2}(R))}\leq \frac{1}{\kappa}\|f\|_{L^{2}(0,T;\dot{H}^{-\frac{\alpha}{2}}(R))}
\end{equation}
due to $u_{nm}(t)$ is T-periodic.
Similarly, multiplying \eqref{e214} by $-b(n,m)_{ij}\frac{2i\pi}{T}$, multiplying \eqref{e215} by $a(n,m)_{ij}\frac{2i\pi}{T}$ and adding them together yield
$$
\int^{T}_{0}\int_{R}|\partial_{t}u_{nm}|^{2}dxdt+\kappa\int^{T}_{0}\int_{R}\Lambda^{\frac{\alpha}{2}}u_{nm}\partial_{t}\Lambda^{\frac{\alpha}{2}}u_{nm}dxdt
=\int^{T}_{0}\int_{R}f\partial_{t}u_{nm}dxdt.
$$
The T-periodicity of $\Lambda^{\frac{\alpha}{2}}u_{nm}(t)$ allows that
\begin{equation}\label{e220}
\|\partial_{t}u_{nm}(t)\|_{L^{2}(0,T;L^{2}(R))}\leq \|f\|_{L^{2}(0,T;L^{2}(R))}.
\end{equation}
At last, multiplying \eqref{e214} by $-a(n,m)_{ij}(\frac{2i\pi}{T})^{2}$, multiplying \eqref{e215} by $-b(n,m)_{ij}(\frac{2i\pi}{T})^{2}$ and adding them together yield
$$
\int^{T}_{0}\int_{R}\partial_{t}u_{nm}\partial_{tt}u_{nm}dxdt+\kappa\int^{T}_{0}\int_{R}\Lambda^{\frac{\alpha}{2}}u_{nm}\partial_{tt}\Lambda^{\frac{\alpha}{2}}u_{nm}dxdt
=\int^{T}_{0}\int_{R}f\partial_{tt}u_{nm}dxdt.
$$
There also follows from the T-periodicity of $\Lambda^{\frac{\alpha}{2}}u_{nm}(t), \partial_{t}u_{nm}(t)$ and $f(t)$ that
\begin{equation}\label{e221}
\|\Lambda^{\frac{\alpha}{2}}(\partial_{t}u_{nm}(t))\|_{L^{2}(0,T;L^{2}(R))}\leq \frac{1}{\kappa}\|f_{t}\|_{L^{2}(0,T;\dot{H}^{-\frac{\alpha}{2}}(R))}.
\end{equation}
Therefore, the uniform bounds \eqref{e219}-\eqref{e221} (for $n$ fixed) imply that there exist a function $u_{n}\in H^{1}(0,T;\dot{H}^{-\frac{\alpha}{2}}(R))$ with $\partial_{t}u_{n}\in L^{2}(0,T;L^{2}(R))$ such that
\begin{equation}\label{e222}
\left\{ \begin{array}{lll}
\Lambda^{\frac{\alpha}{2}}u_{nm}\rightharpoonup \Lambda^{\frac{\alpha}{2}}u_{n}~~ in~~L^{2}(0,T;L^{2}(R)), \\
\partial_{t}u_{nm}\rightharpoonup \partial_{t}u_{n}~~ in~~L^{2}(0,T;L^{2}(R)), \\
\Lambda^{\frac{\alpha}{2}}(\partial_{t}u_{nm})\rightharpoonup \Lambda^{\frac{\alpha}{2}}(\partial_{t}u_{n})~~ in~~L^{2}(0,T;L^{2}(R)).
 \end{array} \right.
\end{equation}
Hence it follows from \eqref{e222} that
\begin{equation}\label{e223}
\int^{T}_{0}\int_{R}\partial_{t}u_{n}\varphi dxdt+\kappa\int^{T}_{0}\int_{R}\Lambda^{\frac{\alpha}{2}}u_{n}\Lambda^{\frac{\alpha}{2}}\varphi dxdt=\int^{T}_{0}\int_{R}f\varphi dxdt,
\end{equation}
for all $\varphi\in C^{\infty}_{0,T}((0,T)\times R)$.

{\bf Step 2: Periodicity and uniform estimates of sequence $\{u_{n}\}$}

To prove that each $u_{n}$ is T-periodic in $H^{\frac{\alpha}{2}}(R)$, let's define the subspace
$$
S_{m}:=span\{\Lambda^{\frac{\alpha}{2}}\psi_{0}, \Lambda^{\frac{\alpha}{2}}\psi_{1},...,\Lambda^{\frac{\alpha}{2}}\psi_{m} \}\subset L^{2}(R).
$$
Then it's easy to see that the embedding
$$
H^{1}(0,T;S_{m})\hookrightarrow C(0,T;S_{m})
$$
is compact due to the finite dimension of $S_{m}$. Then the estimate \eqref{e221} implies that the subsequence $\{u_{nm}\}_{m\in N}$ can be chosen such that $\Lambda^{\frac{\alpha}{2}}u_{nm}(t)\rightarrow \Lambda^{\frac{\alpha}{2}}u_{n}(t)$ uniformly in $[0,T]$. The T-periodicity of $u_{nm}(t)$ shows that
$u_{n}(0)=u_{n}(T)$ in $\dot{H}^{\frac{\alpha}{2}}(R)$. Similarly, the estimate \eqref{e220} can yield that $u_{n}(0)=u_{n}(T)$ in $L^{2}(R)$. Thus, $u_{n}(t,x)$ is T-periodic in $H^{\frac{\alpha}{2}}(R)$ with the form of \eqref{e211}.

The similar process in step 1 can be used on sequence $\{u_{n}\}$, which gives
\begin{equation}\label{e224}
\left\{ \begin{array}{lll}
\|\Lambda^{\frac{\alpha}{2}}u_{n}\|_{L^{2}(0,T;L^{2}(R))}\leq \frac{1}{\kappa}\|f\|_{L^{2}(0,T;\dot{H}^{-\frac{\alpha}{2}}(R))}, \\
\|\partial_{t}u_{n}\|_{L^{2}(0,T;L^{2}(R))}\leq \|f\|_{L_{2}(0,T;L^{2}(R))}, \\
\|\Lambda^{\frac{\alpha}{2}}(\partial_{t}u_{n})\|_{L^{2}(0,T;L^{2}(R))}\leq \frac{1}{\kappa} \|f_{t}\|_{L^{2}(0,T;\dot{H}^{-\frac{\alpha}{2}}(R))} .
 \end{array} \right.
\end{equation}
Thus there exist a function $u\in H^{1}(0,T;\dot{H}^{\frac{\alpha}{2}}(R))$ with $u_{t}\in L^{2}(0,T;L^{2}(R))$ such that
\begin{equation}\label{e225}
\left\{ \begin{array}{lll}
\Lambda^{\frac{\alpha}{2}}u_{n}\rightharpoonup \Lambda^{\frac{\alpha}{2}}u~~ in~~L^{2}(0,T;L^{2}(R)), \\
\partial_{t}u_{n}\rightharpoonup \partial_{t}u~~ in~~L^{2}(0,T;L^{2}(R)), \\
\Lambda^{\frac{\alpha}{2}}(\partial_{t}u_{n})\rightharpoonup \Lambda^{\frac{\alpha}{2}}(\partial_{t}u)~~ in~~L^{2}(0,T;L^{2}(R)).
 \end{array} \right.
\end{equation}
Then \eqref{e23} follows from \eqref{e223} and \eqref{e225} and the estimates \eqref{e25}-\eqref{e27} follow from \eqref{e211} and \eqref{e224}.

{\bf Step 3: Periodicity and make-up regularity of $u(t,x)$}

It's sufficient to finish the proof of theorem if we can establish the T-periodicity of $u$ in \eqref{e24} and the make-up regularity of $u$ in \eqref{e28} and \eqref{e29}. In addition, it's worth noting that the make-up regularity is crucial to deduce the existence and uniqueness of nonlinear problem in next section. Now let's first show $u(t,x)$ is T-periodic. For $\varphi\in C^{\infty}_{0,T}((0,T)\times R)$, it follows from $u(t,x)$ is a weak solution of \eqref{e23} that
\begin{equation}\label{e226}
\int_{R}(u(T)-u(0))\varphi(0)dx-\int^{T}_{0}\int_{R}u\varphi_{t}dxdt+k\int^{T}_{0}\int_{R}\Lambda^{\frac{\alpha}{2}}u \Lambda^{\frac{\alpha}{2}}\varphi dxdt =\int^{T}_{0}\int_{R}f\varphi dxdt.
\end{equation}
On the other hand, \eqref{e223} gives
$$
\int_{R}(u_{n}(T)-u_{n}(0))\varphi(0)dx-\int^{T}_{0}\int_{R}u_{n}\varphi_{t}dxdt+k\int^{T}_{0}\int_{R}\Lambda^{\frac{\alpha}{2}}u_{n} \Lambda^{\frac{\alpha}{2}}\varphi dxdt =\int^{T}_{0}\int_{R}f\varphi dxdt.
$$
Combining the T-periodicity of $u_{n}(t)$ and \eqref{e225} with this formula, we have
\begin{equation}\label{e227}
-\int^{T}_{0}\int_{R}u\varphi_{t}dxdt+k\int^{T}_{0}\int_{R}\Lambda^{\frac{\alpha}{2}}u \Lambda^{\frac{\alpha}{2}}\varphi dxdt =\int^{T}_{0}\int_{R}f\varphi dxdt.
\end{equation}
Thus, $u(t,x)$ is T-periodic by comparing \eqref{e226} with \eqref{e227} due to the arbitrary of $\varphi$.

Now we are in the position to give the proof of \eqref{e28} and \eqref{e29} by using the techniques of Fourier transform. Taking the Fourier transform of \eqref{e21}, we have
\begin{equation}\label{e228}
\partial_{t}\hat{u}(t,\xi)+\kappa|\xi|^{\alpha}\hat{u}(t,\xi)=\hat{f}(t,\xi).
\end{equation}
It's known that $\||\xi|^{\frac{\alpha}{2}}\hat{u}\|_{L^{2}(0,T;L^{2}(R))}$ is well defined from \eqref{e25}-\eqref{e27}. Choose any $r>0$, it's obvious that $\|\hat{u}\|_{L^{2}(0,T;L^{2}([-r,r]^{c}))}$ and $\||\xi|^{\alpha}\hat{u}\|_{L^{2}(0,T;L^{2}([-r,r]))}$ are also well-defined due to
\begin{equation}\label{e229}
\|\hat{u}\|_{L^{2}(0,T;L^{2}([-r,r]^{c}))}=\||\xi|^{-\frac{\alpha}{2}}|\xi|^{\frac{\alpha}{2}}\hat{u}\|_{L^{2}(0,T;L^{2}([-r,r]^{c}))}\leq r^{-\frac{\alpha}{2}}\||\xi|^{\frac{\alpha}{2}}\hat{u}\|_{L^{2}(0,T;L^{2}(R))}
\end{equation}
and
\begin{equation}\label{e230}
\||\xi|^{\alpha}\hat{u}\|_{L^{2}(0,T;L^{2}([-r,r]))}=\||\xi|^{\frac{\alpha}{2}}|\xi|^{\frac{\alpha}{2}}\hat{u}\|_{L^{2}(0,T;L^{2}([-r,r]))}\leq r^{\frac{\alpha}{2}}\||\xi|^{\frac{\alpha}{2}}\hat{u}\|_{L^{2}(0,T;L^{2}(R))}.
\end{equation}
Taking the square on both sides of \eqref{e228}, we can obtain
\begin{equation}\label{e231}
|\partial_{t}\hat{u}|^{2}+\kappa^{2}|\xi|^{2\alpha}|\hat{u}|^{2}+\kappa\partial_{t}(|\xi|^{\alpha}|\hat{u}|^{2})=|\hat{f}|^{2}.
\end{equation}
It follows from (\ref{e231}) that
$$
\kappa^{2}|\hat{u}|^{2}+\kappa\partial_{t}(|\xi|^{-\alpha}|\hat{u}|^{2})\leq|\xi|^{-2\alpha}|\hat{f}|^{2}.
$$
Thanks to \eqref{e229} and the T-periodicity of $u(t)$, integrating both sides of above inequality over $[0,T]\times [-r,r]^{c}$ yields
\begin{equation}\label{e232}
\int^{T}_{0}\|\hat{u}\|^{2}_{L^{2}([-r,r]^{c})}dt\leq \frac{1}{\kappa^{2}}\int^{T}_{0}\||\xi|^{-\alpha}\hat{f}\|^{2}_{L^{2}([-r,r]^{c})}dt\leq \frac{1}{\kappa^{2}}\|f\|^{2}_{L^{2}(0,T;\dot{H}^{-\alpha}(R))},
\end{equation}
then \eqref{e28} follows by letting $r\rightarrow 0$ in \eqref{e232}.

Similarly, we also can deduce that from \eqref{e231}
$$
\kappa^{2}|\xi|^{2\alpha}|\hat{u}|^{2}+\kappa\partial_{t}(|\xi|^{\alpha}|\hat{u}|^{2})\leq |\hat{f}|^{2}.
$$
Thanks to \eqref{e230} and the T-periodicity of $u(t)$, integrating both sides of above inequality over $[0,T]\times [-r,r]$ yields
\begin{equation}\label{e233}
\int^{T}_{0}\||\xi|^{\alpha}\hat{u}\|^{2}_{L^{2}([-r,r])}dt\leq \frac{1}{\kappa^{2}}\int^{T}_{0}\|\hat{f}\|^{2}_{L^{2}([-r,r])}dt\leq \frac{1}{\kappa^{2}}\|f\|^{2}_{L^{2}(0,T;L^{2}(R))},
\end{equation}
then \eqref{e29} follows by letting $r\rightarrow \infty$ in \eqref{e233}. It's obvious that the solution to linearized problem \eqref{e21} is unique, then the proof is completed.
\end{proof}

\section{The Proof of the Theorem \ref{thm1.1}}

In this section, a contraction mapping argument would be constructed to establish the existence and uniqueness of T-periodic solution to \eqref{e11}.
\begin{proof}
For convenience, let's first introduce two spaces
$$
X_{T}:=\{\phi\in X|~\phi(0)=\phi(T)~in ~\dot{H}^{-\frac{\alpha}{2}}(R)\}
$$
and
$$
Y_{T}:=\{\phi\in Y=H^{1}(0,T;H^{\frac{\alpha}{2}}(R))\cap L^{2}(0,T;\dot{H}^{\alpha}) |~\phi(0)=\phi(T)~in ~H^{\frac{\alpha}{2}}(R)\}.
$$
By Gagliardo-Nirenberg inequality and Kato-Ponce type commutator estimates in \cite{Kato}, we will show that $vv_{x}\in X_{T}$ for any $v\in Y_{T}$. It's easy to see that $(vv_{x})(0)=(vv_{x})(T)$ if $v\in Y_{T}$. In addition, we can use the interpolation theorem to obtain
\begin{align}\label{e31}
(\int^{T}_{0}\|vv_{x}\|^{2}_{L^{2}}dt)^{\frac{1}{2}}
&=(\int^{T}_{0}\int_{R}v^{2}v_{x}^{2}dxdt)^{\frac{1}{2}}\leq \max_{t\in [0,T]}\|v(t)\|_{L^{\infty}}(\int^{T}_{0}\|v_{x}\|^{2}_{L^{2}}dt)^{\frac{1}{2}} \nonumber\\
&\lesssim \|v\|_{Y}(\int^{T}_{0}\|\Lambda^{\alpha}v\|_{L^{2}}^{\frac{2}{\alpha}}\|v\|_{L^{2}}^{2-\frac{2}{\alpha}}dt)^{\frac{1}{2}} \nonumber\\
&\leq \|v\|_{Y}(\int^{T}_{0}\|\Lambda^{\alpha}v\|_{L^{2}}^{2}dt)^{\frac{1}{2\alpha}}(\int^{T}_{0}\|v\|_{L^{2}}^{2}dt)^{\frac{\alpha-1}{2\alpha}} \nonumber\\
&\leq \|v\|_{Y}^{2},
 \end{align}
\begin{align}\label{e32}
(\int^{T}_{0}\|vv_{x}\|^{2}_{\dot{H}^{-\frac{\alpha}{2}}}dt)^{\frac{1}{2}}
&\leq\frac{1}{2}(\int^{T}_{0}\|\Lambda^{1-\frac{\alpha}{2}}(vv)\|^{2}_{L^{2}}dt)^{\frac{1}{2}}\lesssim (\int^{T}_{0}\|\Lambda^{1-\frac{\alpha}{2}}v\|^{2}_{L^{2}}\|v\|^{2}_{L^{\infty}}dt)^{\frac{1}{2}}  \nonumber\\
&\lesssim \|v\|_{Y}(\int^{T}_{0}\|\Lambda^{1-\frac{\alpha}{2}}v\|^{2}_{L^{2}}dt)^{\frac{1}{2}} \lesssim \|v\|_{Y}(\int^{T}_{0}\|\Lambda^{\frac{\alpha}{2}}v\|^{2-\alpha}_{L^{2}}\|v\|_{L^{2}}^{\alpha}dt)^{\frac{1}{2}} \nonumber\\
&\leq \|v\|_{Y}(\int^{T}_{0}\|\Lambda^{\frac{\alpha}{2}}v\|_{L^{2}}^{2}dt)^{\frac{\alpha}{4}}(\int^{T}_{0}\|v\|_{L^{2}}^{2}dt)^{\frac{2-\alpha}{4}} \nonumber\\
&\leq \|v\|_{Y}^{2}
\end{align}
and
\begin{align}\label{e33}
(\int^{T}_{0}\|\partial_{t}(vv_{x})\|^{2}_{\dot{H}^{-\frac{\alpha}{2}}}dt)^{\frac{1}{2}}
&=(\int^{T}_{0}\|\partial_{x}(vv_{t})\|^{2}_{\dot{H}^{-\frac{\alpha}{2}}}dt)^{\frac{1}{2}}\leq (\int^{T}_{0}\|\Lambda^{1-\frac{\alpha}{2}}(vv_{t})\|^{2}_{L^{2}}dt)^{\frac{1}{2}}  \nonumber\\
&\lesssim(\int^{T}_{0}\|\Lambda^{1-\frac{\alpha}{2}}v_{t}\|_{L^{2}}^{2}\|v\|^{2}_{L^{\infty}}dt+
\int^{T}_{0}\|\Lambda^{1-\frac{\alpha}{2}}v\|_{L^{2}}^{2}\|v_{t}\|^{2}_{L^{\infty}}dt)^{\frac{1}{2}} \nonumber\\
&\lesssim(\|v\|_{Y}^{2}\int^{T}_{0}\|\Lambda^{1-\frac{\alpha}{2}}v_{t}\|_{L^{2}}^{2}dt+
\|v\|_{Y}^{2}\max_{t\in [0,T]}\|\Lambda^{1-\frac{\alpha}{2}}v(t)\|^{2}_{L^{2}})^{\frac{1}{2}} \nonumber\\
&\lesssim\|v\|_{Y}^{2},
\end{align}
where the embedding $H^{\frac{\alpha}{2}}(R)\hookrightarrow L^{\infty}(R)$ and H\"{o}lder inequality are also used.

At last, it remains to estimate $\|vv_{x}\|_{L^{2}(0,T;\dot{H}^{-\alpha}(R))}$, where the restriction on $1<\alpha<\frac{3}{2}$ is required. As $vv_{x}=\frac{1}{2}\partial_{x}(v^{2})$, we take the Fourier transform to obtain
$$
\widehat{vv_{x}}(t,\xi)=\frac{1}{2}i\xi\widehat{v^{2}}(t,\xi)
$$
and
$$
|\widehat{v^{2}}(t,\xi)|\leq \|v^{2}\|_{L^{1}},\quad for~~(t,\xi)\in (0,T)\times R.
$$
Hence, there hold
$$
|\widehat{vv_{x}}(t,\xi)|\leq \frac{1}{2}|\xi|\|v^{2}\|_{L^{1}}
$$
and
$$
|\xi|^{-\alpha}|\widehat{vv_{x}}(t,\xi)|\leq \frac{1}{2}|\xi|^{1-\alpha}\|v^{2}\|_{L^{1}}, \quad a.e~~on ~~(0,T)\times R.
$$
Then we have
\begin{align}\label{e34}
(\int^{T}_{0}\|vv_{x}\|^{2}_{\dot{H}^{-\alpha}}dt)^{\frac{1}{2}}
&=(\int^{T}_{0}\int_{R}|\xi|^{-2\alpha}|\widehat{vv_{x}}(t,\xi)|^{2}d\xi dt)^{\frac{1}{2}} \nonumber\\
&\leq (\frac{1}{4}\int^{T}_{0}\int_{[-1,1]}|\xi|^{2-2\alpha}\|v^{2}\|^{2}_{L^{1}}d\xi dt \nonumber\\
&+\int^{T}_{0}\int_{[-1,1]^{c}}|\xi|^{-2\alpha}|\widehat{vv_{x}}(t,\xi)|^{2}d\xi dt)^{\frac{1}{2}} \nonumber\\
&\leq (\frac{1}{4}\int^{T}_{0}\|v^{2}\|^{2}_{L^{1}}dt\int_{[-1,1]}|\xi|^{2-2\alpha}d\xi
+\int^{T}_{0}\int_{[-1,1]^{c}}|\widehat{vv_{x}}(t,\xi)|^{2}d\xi dt)^{\frac{1}{2}} \nonumber\\
&\leq (\frac{1}{2(3-2\alpha)}\|v\|^{4}_{L^{4}(0,T;L^{2}(R))}+\|vv_{x}\|^{2}_{L^{2}(0,T;L^{2}(R))})^{\frac{1}{2}} \nonumber\\
&\lesssim (\frac{1}{2(3-2\alpha)}\|v\|^{4}_{Y}+\|v\|^{4}_{Y})^{\frac{1}{2}} \nonumber\\
&\lesssim (1+\frac{1}{\sqrt{2(3-2\alpha)}})\|v\|_{Y}^{2}.
\end{align}
From \eqref{e31}-\eqref{e34}, we obtain
\begin{equation}\label{e35}
\|vv_{x}\|_{X}\lesssim (1+\frac{1}{\sqrt{2(3-2\alpha)}})\|v\|^{2}_{Y}.
\end{equation}

Based on the result in section 2, we find that the solution of \eqref{e11} can be regarded as a fixed point of mapping $\Gamma: Y_{T}\rightarrow Y_{T}$, such that for each $v\in Y_{T}$, $u=\Gamma(v)$ is the T-periodic solution of the following problem
\begin{equation}\label{e36}
u_{t}+\kappa\Lambda^{\alpha}u=f-vv_{x},\quad in ~~(0,T)\times R.
\end{equation}
However, it's difficult to establish directly that the mapping $\Gamma$ is a contraction mapping from $Y_{T}$ to $Y_{T}$. In order to achieve the goal, we have to take some restrictions. Let's define a subspace of $Y_{T}$ by
$$
Y^{*}_{T}=\{\phi\in Y_{T}|~\|\phi\|_{Y}\leq 3(1+\frac{1}{\kappa})\|f\|_{X} \}
$$
According to Theorem \ref{thm2.1} and \eqref{e35}, for a given $v\in Y_{T}^{*}$, we can conclude that there exists a unique solution $u\in Y_{T}$ of \eqref{e36}, which also satisfies
\begin{equation}\label{e37}
\|u\|_{Y}\leq (1+\frac{1}{\kappa})(\|f\|_{X}+\|vv_{x}\|_{X})\leq (1+\frac{1}{\kappa})(\|f\|_{X}+\frac{9\sqrt{2(3-2\alpha)}+9}{\sqrt{2(3-2\alpha)}}(1+\frac{1}{\kappa})^{2}\|f\|_{X}^{2}).
\end{equation}

In the following, we will prove that $\Gamma$ is a contraction mapping from $Y^{*}_{T}$ to $Y^{*}_{T}$. For any $v\in Y^{*}_{T}$, it follows from \eqref{e37} and the smallness condition of $\|f\|_{X}$ that
$$
\|\Gamma(v)\|_{Y}=\|u\|_{Y}\leq (1+\frac{1}{\kappa})(\|f\|_{X}+\frac{9\sqrt{2(3-2\alpha)}+9}{\sqrt{2(3-2\alpha)}}(1+\frac{1}{\kappa})^{2}\|f\|_{X}^{2})\leq 3(1+\frac{1}{\kappa})\|f\|_{X},
$$
which means
$$
\Gamma(Y^{*}_{T})\subseteq Y^{*}_{T}.
$$
On the other hand, let $v_{1}, v_{2}\in Y^{*}_{T}$ and $u_{1}=\Gamma(v_{1}), u_{2}=\Gamma(v_{2})$, Theorem \ref{thm2.1} and the smallness condition of $\|f\|_{X}$ yield
\begin{align}\label{e38}
\|\Gamma(v_{1})-\Gamma(v_{2})\|_{Y}
&=\|u_{1}-u_{2}\|_{Y} \nonumber\\
&\leq (1+\frac{1}{\kappa})(\|v_{1}\|_{Y}+\|v_{2}\|_{Y})\|v_{1}-v_{2}\|_{Y} \nonumber\\
&\leq 6(1+\frac{1}{\kappa})^{2}\|f\|_{X}\|v_{1}-v_{2}\|_{Y}<\|v_{1}-v_{2}\|_{Y}.
\end{align}
Thus $\Gamma$ is a contraction mapping and the proof is completed.
\end{proof}
\begin{remark}
Based on the proof of Theorem \ref{thm1.1}, it's obvious that the estimate \eqref{e13} can be improved to
$$
\|u\|_{H^{1}(0,T;H^{\frac{\alpha}{2}}(R))}+\|\Lambda^{\alpha}u\|_{L^{2}(0,T;L^{2}(R))}\leq \beta(1+\frac{1}{\kappa})\|f\|_{X}
$$
for any $\beta>1$.
\end{remark}

\section{The proof of Theorem \ref{thm1.3}}
In this section, we mainly discuss the asymptotic stability of the T-periodic solution $u_{T}(t,x)$ obtained in section 3. It's worth noting that the smallness condition on the external force will play an important role here.
Before proving Theorem \ref{thm1.3}, let's introduce the lemma on existence of viscosity solution of \eqref{e11}, which has been proved in \cite{Xu}.
\begin{lemma}\label{lem4.2}
If $u_{0}\in L^{2}(R)$ and $f\in L^{2}(0,T; \dot{H}^{-\frac{\alpha}{2}}(R)\cap L^{2}(R))$, then there exists a global viscosity weak solution $u\in L^{\infty}([0,T),L^{2}(R))\cap L^{2}([0,T),\dot{H}^{\frac{\alpha}{2}}(R))$ to \eqref{e11}, which satisfies for any given $T>0$,
\begin{equation}
\begin{aligned}
~&-\int^{T}_{0}\langle u, \psi_{t}\rangle dt-\frac{1}{2}\int^{T}_{0}\langle u^{2},\psi_{x}\rangle dt+\kappa\int^{T}_{0}\langle\Lambda^{\frac{\alpha}{2}}u,
\Lambda^{\frac{\alpha}{2}}\psi\rangle dt \nonumber\\
&=\langle u_{0},\psi(x,0)\rangle+\int^{T}_{0}\langle f,\psi\rangle dt, \quad for  ~~ \psi\in C_{c}^{\infty}([0,T)\times R).
 \end{aligned}
\end{equation}
Moreover, there holds
$$
\sup_{t\geq 0}\|u(t)\|_{L^{2}}^{2}+\kappa\int^{T}_{0}\|\Lambda^{\frac{\alpha}{2}} u(t)\|_{L^{2}}^{2}dt\leq \|u_{0}\|_{L^{2}(R)}^{2}+\frac{1}{\kappa}\|f\|_{L^{2}(0,T; \dot{H}^{-\frac{\alpha}{2}}(R)\cap L^{2}(R))}^{2}.
$$
\end{lemma}

Now we are in the position to prove the Theorem \ref{thm1.3}.
\begin{proof}
Let $w(t,x)=u(t,x)-u_{T}(t,x)$, then $w$ satisfies
\begin{equation}\label{e41}
\left\{ \begin{array}{ll}
 w_{t}+ww_{x}+(u_{T}w)_{x}+\kappa\Lambda^{\alpha}w=0,  \\
w_{0}(x)=u_{0}(x)-u_{T}(0,x),
 \end{array} \right.
\end{equation}

Multiplying both sides of \eqref{e42} by $w$ and integrating on $R$ give
\begin{equation}\label{e42}
\frac{1}{2}\frac{d}{dt}\int_{R}|w|^{2}dx+\kappa\int_{R}|\Lambda^{\frac{\alpha}{2}}w|^{2}dx=-\int_{R}(u_{T}w)_{x}wdx.
\end{equation}
On the other hand, there holds
\begin{align}\label{e43}
|\int_{R}(u_{T}w)_{x}wdx|
 &\leq |\int_{R} \Lambda^{1-\frac{\alpha}{2}}(u_{T}w)\Lambda^{\frac{\alpha}{2}}wdx| \nonumber\\
 &\leq \|\Lambda^{1-\frac{\alpha}{2}}(u_{T}w)\|_{L^{2}} \|\Lambda^{\frac{\alpha}{2}}w\|_{L^{2}}  \nonumber\\
 &\leq (\|\Lambda^{1-\frac{\alpha}{2}}u_{T}\|_{L^{2}}\|w\|_{L^{\infty}}+\|u_{T}\|_{L^{\frac{1}{\alpha-1}}}
\|\Lambda^{1-\frac{\alpha}{2}}w\|_{L^{\frac{2}{3-2\alpha}}})\|\Lambda^{\frac{\alpha}{2}}w\|_{L^{2}}  \nonumber\\
 &\leq (\|\Lambda^{1-\frac{\alpha}{2}}u_{T}\|_{L^{2}}+\|u_{T}\|_{L^{\frac{1}{\alpha-1}}})\|\Lambda^{\frac{\alpha}{2}}w\|_{L^{2}}^{2} \nonumber\\
 &\leq 2\|u_{T}\|_{H^{\frac{\alpha}{2}}}\|\Lambda^{\frac{\alpha}{2}}w\|_{L^{2}}^{2}.
\end{align}
Therefore \eqref{e42} and \eqref{e43} imply
\begin{equation}\label{e44}
\frac{d}{dt}\int_{R}|w|^{2}dx+2\kappa\int_{R}|\Lambda^{\frac{\alpha}{2}}w|^{2}dx \leq 4\|u_{T}\|_{H^{\frac{\alpha}{2}}}\|\Lambda^{\frac{\alpha}{2}}w\|_{L^{2}}^{2}.
 \end{equation}
Integrating \eqref{e44} with respect to time variable over $(0,T)$, we obtain
  \begin{equation}\label{e45}
 \int_{R}|w|^{2}dx+2\kappa\int_{0}^{T}\int_{R}|\Lambda^{\frac{\alpha}{2}}w|^{2}dxdt\leq 4\max_{t\in[0,T]}\|u_{T}\|_{H\frac{\alpha}{2}}\int_{0}^{T}\int_{R}|\Lambda^{\frac{\alpha}{2}}w|^{2}dxdt+\int_{R}|w_{0}|^{2}dx.
 \end{equation}
Thus \eqref{e14} follows from \eqref{e45}, \eqref{e13} and the smallness of $\|f\|_{X}\leq\frac{\kappa^{2}}{6(\kappa+1)^{2}}< \frac{\kappa^{2}}{6(\kappa+1)}$.
\end{proof}

    \bigskip

\noindent \textbf{Acknowledgement.} {This project is supported by National Natural Science Foundation of China (No:11571057).}

\end{document}